\documentclass[11pt]{article}
\usepackage{amsmath,amsthm}
\usepackage{amssymb}
\usepackage{enumerate}
\usepackage[mathscr]{eucal}
\usepackage[cm]{fullpage}
\usepackage[english]{babel}
\usepackage[latin1]{inputenc}
\newcommand{\T}{\mathscr{T}}
\newcommand{\End}{\mathrm{End}}
\newcommand{\PEnd}{\mathrm{PEnd}}
\newcommand{\wPEnd}{\mathrm{wPEnd}}
\newcommand{\wEnd}{\mathrm{wEnd}}

\newcommand{\Aut}{\mathrm{Aut}}
\newcommand{\IEnd}{\mathrm{IEnd}}
\newcommand{\PAut}{\mathrm{PAut}}
\newcommand{\PT}{\mathscr{PT}}
\newcommand{\Sym}{\mathscr{S}}
\newcommand{\I}{\mathscr{I}}

\theoremstyle{plain}
\newtheorem{theorem}{Theorem}[section]
\newtheorem{proposition}[theorem]{Proposition}
\newtheorem{lemma}[theorem]{Lemma}
\newtheorem{corollary}[theorem]{Corollary}

\theoremstyle{remark}

\def\im{\mathop{\mathrm{Im}}\nolimits}
\def\dom{\mathop{\mathrm{Dom}}\nolimits}
\def\rank{\mathop{\mathrm{rank}}\nolimits}
\def\N{\mathbb N}
\def\id{\mathop{\mathrm{id}}\nolimits}

\numberwithin{equation}{section}

\begin{document}

\title{Partial Automorphisms and Injective Partial Endomorphisms of a Finite Undirected Path\footnote{
This work is funded by national funds through the FCT - Funda\c c\~ao para a Ci\^encia e a Tecnologia, I.P., under the scope of the project UIDB/00297/2020 (Center for Mathematics and Applications).}}

\author{I. Dimitrova, V.H. Fernandes, J. Koppitz and T.M. Quinteiro}

\maketitle

\renewcommand{\thefootnote}{}



\renewcommand{\thefootnote}{\arabic{footnote}}
\setcounter{footnote}{0}

\begin{abstract}
In this paper, we study partial automorphisms and, more generally, injective partial endomorphisms of a finite undirected path from Semigroup Theory perspective. Our main objective is to give formulas for the ranks of the monoids $\IEnd(P_n)$ and $\PAut(P_n)$ of all injective partial endomorphisms and of all partial automorphisms of the undirected path $P_n$ with $n$ vertices. We also describe Green's relations of $\PAut(P_n)$ and $\IEnd(P_n)$ and calculate their cardinals.
\end{abstract}

\section*{Introduction and Preliminaries}

As well as automorphisms of graphs allow to establish natural connections between Graph Theory and Group Theory, endomorphisms of graphs allow similar connections between Graph Theory and Semigroup Theory. Likewise, in particular, partial automorphisms of graphs relates Graph Theory with Inverse Semigroup Theory. This has led, over the last decades, many authors to become interested in the study of combinatorial and algebraic properties of monoids of endomorphisms of graphs. One of the most studied algebraic notions is
\textit{regularity}, in the sense of Semigroup Theory.
A general solution to the problem, posed in 1988 by M\'arki \cite{Marki:1988}, of which graphs have a regular monoid of endomorphisms proved very difficult to obtain. Despite that, various authors studied and solved this question for some special classes of graphs (for instance, see  \cite{Fan:1993,Fan:1997,Fan:2002,Hou&Gu:2016,Hou&Gu&Shang:2014,Hou&Gu&Shang:2015,Hou&Luo&Fan:2012,Hou&Song&Gu:2017,Knauer&Wanichsombat:2014,Li:2006,Li&Chen:2001,Lu&Wu:2013,Pipattanajinda&Knauer&Gyurov&Panma:2016,Wilkeit:1996}).

\smallskip

The \textit{rank} of a monoid $S$, denoted by $\rank S$, is the least number of generators of $S$.
In this paper, we focus our attention on this important notion of Semigroup Theory, which has been, in recent years, the subject of intensive research.

\smallskip

Let $\Omega$ be a finite set
with at least $2$ elements.
It is well-known that the symmetric group $\Sym(\Omega)$ of $\Omega$
has rank $2$ (as a semigroup, a monoid or a group). Furthermore, the monoid of all transformations $\T(\Omega)$ of $\Omega$,
the monoid of all partial transformations $\PT(\Omega)$ of $\Omega$ and the symmetric inverse monoid $\I(\Omega)$ of $\Omega$ have
ranks $3$, $4$, and $3$, respectively.
The survey \cite{Fernandes:2002} presents
these results and similar ones for other classes of transformation monoids,
in particular, for monoids of order-preserving transformations and
for some of their extensions. More recently, for instance, the papers
\cite{AlKharousi&Kehinde&Umar:2014s,Araujo&al:2015,Cicalo&al:2015,Fernandes&al:2014,Fernandes&al:2016,Fernandes&al:2018ip,Fernandes&Quinteiro:2014,Fernandes&Sanwong:2014,Huisheng:2009,Zhao:2011,Zhao&Fernandes:2015}
are dedicated to the computation of the ranks of certain (classes of transformation) semigroups or monoids.

\medskip

Now, let $G=(V,E)$ be a simple graph (i.e. an undirected graph without loops and without multiple edges).
Let $\alpha$ be a partial transformation of $V$.
Denote by $\dom \alpha$ the domain of $\alpha$ and by $\im \alpha$ the image of $\alpha$.
We say that $\alpha$ is:
\begin{itemize}
\item a \textit{partial endomorphism} of $G$ if $\{u,v\}\in E$ implies  $\{u\alpha,v\alpha\}\in E$, for all $u,v\in\dom \alpha$;
\item a \textit{weak partial endomorphism} of $G$ if $\{u,v\}\in E$ and $u\alpha\ne v\alpha$ imply  $\{u\alpha,v\alpha\}\in E$, for all $u,v\in\dom \alpha$;
\item a \textit{partial automorphism} of $G$ if $\alpha$ is an injective mapping (i.e. a partial permutation) and  $\alpha$ and $\alpha^{-1}$ are both partial endomorphisms;

\item if $\alpha$ is a full mapping (i.e. $\alpha\in\T(V)$) then to a partial endomorphism (respectively, weak partial endomorphism and partial automorphism) we just call \textit{endomorphism} (respectively, \textit{week endomorphism} and \textit{automorphism}).
\end{itemize}

Notice that, any partial endomorphism is a weak partial endomorphism and any injective weak partial endomorphism is an (injective) partial endomorphism. Moreover, for finite graphs, any bijective endomorphism is an automorphism.

\medskip

Denote by:
\begin{itemize}
\item $\wPEnd(G)$ the set of all weak partial endomorphisms of $G$;
\item $\PEnd(G)$ the set of all partial endomorphisms of $G$;
\item $\End(G)$ the set of all endomorphisms of $G$;
\item $\wEnd(G)$ the set of all weak endomorphisms of $G$;
\item $\IEnd(G)$ the set of all injective partial endomorphisms of $G$;
\item $\PAut(G)$ the set of all partial automorphisms of $G$;
\item $\Aut(G)$ the set of all automorphisms of $G$.
\end{itemize}

Clearly, $\wPEnd(G)$, $\PEnd(G)$, $\End(G)$, $\wEnd(G)$, $\IEnd(G)$, $\PAut(G)$, and $\Aut(G)$ are monoids under composition of
maps with the identity mapping $\id$ as the identity element. Moreover, $\Aut(G)$ is also a group and $\PAut(G)$ is an inverse semigroup:
$\Aut(G) \subseteq \Sym(V)$ and $\PAut(G) \subseteq \I(V)$.
It is also clear that
$$
\Aut(G) \subseteq \End(G) \subseteq \wEnd(G) \subseteq \wPEnd(G)
\quad\text{and}\quad
\Aut(G) \subseteq \PAut(G) \subseteq \IEnd(G) \subseteq \PEnd(G) \subseteq \wPEnd(G)
$$
(these inclusions may not be strict).

\medskip

Let $\N$ be the set of all natural numbers greater than zero and let $n\in\N$.
Let $P_n$ be the undirected path with $n$ vertices. Notice that we may take
$$
P_n=\left(\{1,\ldots,n\},\{\{i,i+1\}\mid i=1,\ldots,n-1\}\right).
$$

The number of endomorphisms of $P_n$ has been determined by Arworn
\cite{Arworn:2009} (see also the paper \cite{Michels&Knauer:2009} by Michels and Knauer).
In addition, several other combinatorial and algebraic properties of $P_n$ were also studied in these two papers and also,
for instance, in \cite{Arworn&Knauer&Leeratanavalee:2008,Hou&Luo&Cheng:2008}.
The authors in \cite{Dimitrova&Fernandes&Koppitz&Quinteiro:2019online} studied several properties of $\End(P_n)$ and $\wEnd(P_n)$. In particular, they characterized regular elements of $\End(P_n)$ and $\wEnd(P_n)$, determined the cardinal of $\wEnd(P_n)$ and, for $n\ge2$,  showed that
$$
\rank \wEnd(P_n) = n+\sum_{j=1}^{\lfloor\frac{n-3}{3}\rfloor}\lfloor\frac{n-3j-1}{2}\rfloor
\quad\text{and}\quad
\rank \End(P_n) = 1+\lfloor\frac{n-1}{2}\rfloor+\sum_{j=1}^{\lfloor\frac{n-3}{3}\rfloor}\lfloor\frac{n-3j-1}{2}\rfloor.
$$
Notice that $P_{1}=(\{1\}, \emptyset)$. Thus, for $n=1$, both the monoids $\End(P_n)$ and $\wEnd(P_n)$ are trivial and so naturally they have (monoid) rank equal to zero.

\medskip

 The main objective of the present paper is to determine the ranks of the monoids $\PAut(P_n)$ and $\IEnd(P_n)$. We will show that
 $$
 \rank \PAut(P_n) =
 \left\{
 \begin{array}{ll}
 2 & \mbox{for $n=1$} \\
 2 & \mbox{for $n=2$} \\
 3 & \mbox{for $n=3$} \\
 n-1 & \mbox{for $n\ge4$}
 \end{array}
 \right.
 \qquad\text{and}\qquad
  \rank \IEnd(P_n) =
 \left\{
 \begin{array}{ll}
 2 & \mbox{for $n=1$} \\
 2 & \mbox{for $n=2$} \\
 4 & \mbox{for $n=3$} \\
 n+\lceil\frac{n}{2}\rceil-2 & \mbox{for $n\ge4$}.
 \end{array}
 \right.$$

We also aim to describe Green's relations of $\PAut(P_n)$ and $\IEnd(P_n)$ and to calculate the cardinals of both monoids.

\medskip

Observe that $\PAut(P_n)$ and $\IEnd(P_n)$ are submonoids of the symmetric inverse monoid $\I_n=\I(\{1,\ldots,n\})$.

\medskip

Recall that the Green's relations $\mathcal{L}$, $\mathcal{R}$, and $\mathcal{J}$ of a
monoid $S$ are defined as following: for $\alpha, \beta \in S$,
\begin{itemize}
\item $\alpha \mathcal{L} \beta$ if and only if there exist $\gamma, \delta \in S$ such that $\alpha = \gamma\beta$ and $\beta = \delta\alpha$;

\item $\alpha \mathcal{R} \beta$ if and only if there exist $\gamma', \delta' \in S$ such that $\alpha = \beta\gamma'$ and $\beta = \alpha\delta'$;

\item $\alpha \mathcal{J} \beta$ if and only if there exist $\gamma, \gamma', \delta, \delta' \in S$ such that $\alpha = \gamma\beta\gamma'$ and $\beta = \delta\alpha\delta'$.
\end{itemize}
The relations $\mathcal{L}$ and $\mathcal{R}$ commute (i.e. $ \mathcal{L} \circ \mathcal{R} = \mathcal{R} \circ \mathcal{L}$)
and the Green's relation $\mathcal{D}$ is defined by
$\mathcal{D} = \mathcal{L} \circ \mathcal{R} = \mathcal{R} \circ \mathcal{L}$ (i.e. $\alpha \mathcal{D} \beta$ if and only if there exists $\sigma \in S$
such that $\alpha \mathcal{L} \sigma \mathcal{R} \beta$, for $\alpha, \beta \in S$).
Notice that for a finite monoid the relations $\mathcal{J}$ and $\mathcal{D}$ coincide.
Finally, we have the Green's relation $\mathcal{H}$ defined by $\mathcal{H} = \mathcal{L} \cap \mathcal{R}$.

\medskip

If $S$ is an inverse semigroup of injective partial transformations on a given set,
then the relations $\mathcal{L}$, $\mathcal{R}$, and $\mathcal{H}$
can be described as following: for $\alpha, \beta \in S$,
\begin{itemize}
\item $\alpha \mathcal{L} \beta$ if and only if $\im \alpha = \im \beta$;

\item $\alpha \mathcal{R} \beta$ if and only if $\dom \alpha = \dom \beta$;

\item $\alpha \mathcal{H} \beta $ if and only if $\im \alpha = \im \beta $ and $\dom \alpha = \dom \beta$.
\end{itemize}

Since $\PAut(P_n)$ is an inverse semigroup, it remains to obtain a description of its Green's relation $\mathcal{J}$.
On the other hand, that is not the situation of $\IEnd(P_n)$, for $n \geq 3$, since $\IEnd(P_n)$ is not an inverse semigroup
(for instance,
$\left(
\begin{array}{cc}
1&3 \\
1&2
\end{array}\right)\in\IEnd(P_n)$ is not a regular element of $\IEnd(P_n)$).
Notice that, $\IEnd(P_n) = \PAut(P_n)$, for $n=1, 2$.

\medskip

For general background on Semigroup Theory and standard notation, we refer the reader to Howie's book \cite{Howie:1995}.
Regarding Algebraic Graph Theory, our main reference is Knauer's book \cite{Knauer:2011}.

\section{Green's Relations}

Let $n \in \N$. We now describe the Green's relations $\mathcal{L}$, $\mathcal{R}$, $\mathcal{H}$, and $\mathcal{J}$
of the monoid $\IEnd(P_{n})$ as well as the Green's relation $\mathcal{J}$ of the inverse monoid $\PAut(P_{n})$.

In this section, for a set $X \subseteq \N$, we need the following concept.
A set $I\subseteq X$ is called a \textit{maximal interval} of $X$ if $I$ satisfies the
following properties:
\begin{itemize}
  \item $I$ is an interval of $X$ (i.e. $x,y\in I$ and $z\in \N$ with $x<z<y$ implies $z\in I$);
  \item  if $J\subseteq X$ is an interval of $X$ then $I\subseteq J$ implies $I=J$.
\end{itemize}

Recall that a partial transformation $\alpha$ of $\{1,\ldots,n\}$ is said to be \textit{order-preserving} (respectively, \textit{order-reversing})
if $x<y$ implies $x\alpha\leq y\alpha$ (respectively, if $x<y$ implies $x\alpha\geq y\alpha$), for all $x,y\in\dom\alpha$.

\medskip

Let $\alpha\in\I_n$. The following observations are easy to show:
\begin{itemize}
\item $\alpha\in\IEnd(P_{n})$ if and only if for each interval $I$ of $\dom\alpha$ the image $I\alpha$ is an interval of $\im\alpha$;

\item if $\alpha\in\IEnd(P_{n})$ then $\alpha$ is order-preserving or order-reversing in $I$
(i.e. the restriction $\alpha|_I$ of $\alpha$ to $I$ is an order-preserving or order-reversing transformation),
for each interval $I$ of $\dom\alpha$;

\item if for each maximal interval $I$ of $\dom\alpha$ the image $I\alpha$ is an interval of $\im\alpha$ and $\alpha$ is order-preserving or order-reversing in $I$ then $\alpha\in\IEnd(P_{n})$;

\item if $\alpha\in\PAut(P_n)$ and $I$ is a maximal interval of $\dom\alpha$ then the image $I\alpha$ is a maximal interval of $\im\alpha$;

\item if for each maximal interval $I$ of $\dom\alpha$ the image $I\alpha$ is a maximal interval of $\im\alpha$ and $\alpha$ is order-preserving or order-reversing in $I$ then $\alpha\in\PAut(P_{n})$.
\end{itemize}

Let $\alpha\in\I_n$. Let $\{X_1,\ldots,X_k\}$ be a partition of $\dom\alpha$.
We will use the notation $\alpha=\left(
                      \begin{array}{ccc}
                        X_1 & \cdots & X_k \\
                        Y_1& \cdots & Y_k \\
                      \end{array}
                    \right)$
to express that $Y_i=(X_i)\alpha$, for $i\in\{1,\ldots,k\}$.

\medskip

Let $\alpha, \beta \in \IEnd(P_n)$.
Since $\IEnd(P_{n})$ is a submonoid of the inverse monoid $\I_{n}$, if $\alpha \mathcal{L} \beta $
(respectively, $\alpha\mathcal{R} \beta$) in $\IEnd(P_{n})$ then $\alpha \mathcal{L} \beta$ (respectively, $\alpha\mathcal{R} \beta$) in $\I_{n}$, whence
$\im \alpha = \im \beta$ (respectively, $\dom \alpha = \dom \beta$).
Moreover, we have the following descriptions of the relations $\mathcal{L}$ and $\mathcal{R}$ in $ \IEnd(P_n)$:

\begin{proposition}\label{prGR1}
Let $\alpha, \beta \in \IEnd(P_n)$ and let $\{I_1, I_2, \ldots, I_k\}$ and $\{I_1', I_2', \ldots, I_l'\}$ be the $($partitions into$)$ maximal intervals of $\dom \alpha$
and in $\dom \beta$, respectively. Then, the following three conditions are equivalent:
\begin{enumerate}
  \item $\alpha {\cal L} \beta$;
  \item $\{I_1\alpha, I_2\alpha, \ldots, I_k\alpha\} = \{I_1'\beta, I_2'\beta, \ldots, I_l'\beta\}$;
  \item $\im\alpha = \im\beta$ and $\alpha\beta^{-1} \in \PAut(P_n)$.
\end{enumerate}
\end{proposition}

\begin{proof}$\nonumber$  [$1 \Rightarrow 2$]
Suppose that $\alpha {\cal L} \beta$. Then, by the definition of the Green's relation $\cal L$,
there exist $\gamma, \delta \in \IEnd(P_n)$ such that $\alpha = \gamma\beta$ and $\beta = \delta\alpha$.
Let $i \in \{1,\ldots,k\}$. Since $\dom\alpha\subseteq\dom\gamma$, $I_i$ is also an interval of $\dom\gamma$,
whence $I_i\gamma$ is an interval of $\dom\beta$ and so $I_i\gamma\subseteq I_j'$, for some $j \in \{1,\ldots,l\}$.
It follows that $I_i\alpha=I_i\gamma\beta\subseteq I_j'\beta$, for some $j \in \{1,\ldots,l\}$.
Similarly, we may show that, for all $j \in \{1,\ldots,l\}$, there exists $i \in \{1,\ldots,k\}$ such that $I_j'\beta \subseteq I_i\alpha$.
Now, since $\im\alpha = \im\beta$, we may  deduce that $\{I_1\alpha, I_2\alpha, \ldots, I_k\alpha\} = \{I_1'\beta, I_2'\beta, \ldots, I_l'\beta\}$.

\medskip

[$2 \Rightarrow 3$] From $\{I_1\alpha, I_2\alpha, \ldots, I_k\alpha\} = \{I_1'\beta, I_2'\beta, \ldots, I_l'\beta\}$ it follows immediately
that $k=l$ and $\im\alpha = \im\beta$. Let $\sigma$ be the permutation of $\{1,\ldots,k\}$ such that $I_i\alpha = I_{i\sigma}'\beta$, for all $i \in \{1,\ldots,k\}$. Then
$\alpha\beta^{-1} = \left(
                      \begin{array}{cccc}
                        I_1 & I_2 & \cdots & I_k \\
                        I_{1\sigma}' & I_{2\sigma}' & \cdots & I_{k\sigma}' \\
                      \end{array}
                    \right)$
and so  $\alpha\beta^{-1}$ maps maximal intervals of its domain into maximal intervals of its image.
Hence, in order to prove that $\alpha\beta^{-1} \in \PAut(P_n)$,
it suffices to show that $\alpha\beta^{-1}$ is order-preserving or order-reversing in $I_i$,
for $i \in \{1,\ldots,k\}$. Let $i \in \{1,\ldots,k\}$. Then, we have
$
\alpha\beta^{-1}|_{I_i} = \alpha|_{I_i}\beta^{-1}|_{I_i\alpha}=\alpha|_{I_i}\beta^{-1}|_{I_{i\sigma}'\beta}
$
and
$
\beta^{-1}|_{I_{i\sigma}'\beta}= \left(
                      \begin{array}{c}
                       I_{i\sigma}'\beta \\
                       I_{i\sigma}'\\
                      \end{array}
                    \right)
$.
As $I_i$ is an interval, $\alpha|_{I_i}$  is order-preserving or order-reversing. On the other hand,
as $I_{i\sigma}'$ is an interval, $\beta|_{I_{i\sigma}'}$ is order-preserving or order-reversing and so its inverse mapping
$\beta^{-1}|_{I_{i\sigma}'\beta}$ is also order-preserving or order-reversing. Thus, $\alpha\beta^{-1}$ is order-preserving or order-reversing in $I_i$,
as required.

\medskip

[$3 \Rightarrow 1$] From $\im\alpha = \im\beta$ and $\alpha\beta^{-1} \in \PAut(P_n)$,
 it follows that $\alpha\beta^{-1},  \beta\alpha^{-1}=(\alpha\beta^{-1})^{-1}\in \IEnd(P_n)$,
 $(\alpha\beta^{-1})\beta = \alpha(\beta^{-1}\beta) = \alpha\id|_{\im\beta} = \alpha\id|_{\im\alpha} = \alpha$
 and $(\alpha\beta^{-1})^{-1}\alpha = (\beta\alpha^{-1})\alpha = \beta(\alpha^{-1}\alpha) = \beta\id|_{\im\alpha} = \beta\id|_{\im\beta} = \beta$,
 whence $\alpha {\cal L} \beta$.
\end{proof}

\begin{proposition}\label{prGR2}
Let $\alpha, \beta \in \IEnd(P_n)$. Then $\alpha {\cal R} \beta$ if and only if $\dom\alpha = \dom\beta$ and $\alpha^{-1}\beta \in \PAut(P_n)$.
\end{proposition}
\begin{proof}
Suppose that $\alpha {\cal R} \beta$. Then $\dom\alpha = \dom\beta$. Moreover, there exist transformations $\gamma, \delta \in \IEnd(P_n)$
such that $\beta = \alpha\gamma$ and $\alpha = \beta\delta$. Then, we have
$\alpha^{-1}\beta = \alpha^{-1}\alpha\gamma = \id|_{\im\alpha}\gamma = \gamma|_{\im\alpha}$
and
$(\alpha^{-1}\beta)^{-1}=\beta^{-1}\alpha = \beta^{-1}\beta\delta = \id|_{\im\beta}\delta = \delta|_{\im\beta}$.
Since, clearly, any restriction of a transformation of $\IEnd(P_n)$ is still a transformation of $\IEnd(P_n)$,
we have $\alpha^{-1}\beta, \beta^{-1}\alpha\in\IEnd(P_n)$ and so $\alpha^{-1}\beta \in \PAut(P_n)$.

Conversely,  admit that $\dom\alpha = \dom\beta$ and $\alpha^{-1}\beta \in \PAut(P_n)$.
Then  $\alpha^{-1}\beta,\beta^{-1}\alpha=(\alpha^{-1}\beta)^{-1} \in \IEnd(P_n)$,
$\beta = \id|_{\dom\beta}\beta = \id|_{\dom\alpha}\beta = (\alpha\alpha^{-1})\beta = \alpha(\alpha^{-1}\beta)$
and
$\alpha = \id|_{\dom\alpha}\alpha = \id|_{\dom\beta}\alpha = (\beta\beta^{-1})\alpha = \beta(\beta^{-1}\alpha)$,
whence $\alpha {\cal R} \beta$, as required.
\end{proof}

Since ${\cal H} = {\cal R} \cap {\cal L}$, it follows immediately that:

\begin{corollary}\label{coGR1}
Let $\alpha, \beta \in \IEnd(P_n)$. Then $\alpha {\cal H} \beta$ if and only if $\dom\alpha = \dom\beta$,
$\im\alpha = \im \beta$ and $\alpha^{-1}\beta, \alpha\beta^{-1} \in \PAut(P_n)$.
\end{corollary}

\medskip

Before presenting our descriptions of the Green's relation $\mathcal{J}$ on $\IEnd(P_{n})$ and on $\PAut(P_{n})$, we need to introduce some notions and notations.

\smallskip

For $A, B \subseteq \N$, denote by $A < B$ if $a < b$ for all $a \in A$ and $b \in B$.

\smallskip

Let $a=(a_1,\ldots,a_p)$ be a sequence of elements of $\N$. We define the \textit{reverse of $a$} as being the sequence $a^R=(a_p,\ldots,a_1)$.

\smallskip

Let $\alpha \in \IEnd(P_n)$ and let $J$ be a maximal interval of $\im\alpha$.
Define the \textit{type of $J$} to be the sequence $\tau_\alpha(J) = (|I_{1}|, |I_{2}|, \ldots, |I_{p}|)$,
where $\{I_{1}, I_{2}, \ldots, I_{p}\}$ are the maximal intervals of $J\alpha^{-1}$ such that $I_{i}\alpha < I_{i+1}\alpha$, for $1 \leq i < p$.

Now, let $\alpha,\beta \in \IEnd(P_n)$. We say that $\alpha$ and $\beta$ have \textit{similar type} if there
exists a bijection $\sigma$ from the set of maximal intervals of $\im \alpha$ into the set of maximal intervals of $\im \beta$ such that
$\tau_{\alpha}(J) \in \{\tau_{\beta}(J\sigma), \tau_{\beta}(J\sigma)^{R}\}$, for any maximal interval $J$ of $\im \alpha$.

Observe that two elements $\alpha$ and $\beta$ of  $\IEnd(P_n)$ have similar type if and only if
they have maximal intervals of their images with the same type up to reversion and the same number of occurrences.

\begin{lemma}\label{leGR1}
Let $\alpha, \beta \in \IEnd(P_{n})$ be such that $\alpha$ and $\beta$ have
similar type. Then, there exist $\gamma, \delta \in \PAut(P_{n})$ such that
$\beta = \gamma \alpha \delta$ and $\alpha = \gamma^{-1} \beta \delta^{-1}$.
\end{lemma}
\begin{proof}
Let $\{J_1,J_2,\ldots,J_k\}$ and $\{J_1',J_2',\ldots,J_k'\}$ be the maximal intervals of $\im \alpha$ and $\im \beta$, respectively.
Then there exist a permutation $\sigma$ of $\{1,\ldots,k\}$ such that $\tau_\beta(J_r') \in \{\tau_\alpha(J_{r\sigma}), \tau_\alpha(J_{r\sigma})^R\}$, for $r = 1,\ldots,k$.

For $1 \leq r \leq k$, let $\{I_{r,1}', I_{r,2}', \ldots, I_{r,p_r}'\}$ and $\{I_{r\sigma,1}, I_{r\sigma,2}, \ldots, I_{r\sigma,p_r}\}$ be the maximal intervals of $J_r'\beta^{-1}$ and $J_{r\sigma}\alpha^{-1}$, respectively, such that $I_{r,i}'\beta < I_{r,i+1}'\beta$ and $I_{r\sigma,i}\alpha < I_{r\sigma,i+1}\alpha$, for all $1 \leq i < p_r$.
Moreover, let $J_{r,i}' = I_{r,i}'\beta$ and $J_{r\sigma,i} = I_{r\sigma,i}\alpha$, for $r=1,\ldots,k$ and $i=1,\ldots,p_r$.
Clearly, $J_{r}' = J_{r,1}' \cup J_{r,2}' \cup \cdots \cup J_{r,p_r}'$ and $J_{r\sigma} = J_{r\sigma,1} \cup J_{r\sigma,2} \cup \cdots \cup J_{r\sigma,p_r}$.

\smallskip

Let $r=1,\ldots,k$. We define partial transformations $\gamma_{r}$ and $\delta_{r}$ as following:
\begin{itemize}
\item $\dom\gamma_{r}=\cup\{I'_{r,1},I'_{r,2}, \ldots, I'_{r,p_{r}}\}=J'_r\beta^{-1}$;

\item $\dom\delta_{r}=\cup\{J_{r\sigma,1},J_{r\sigma,2},\ldots, J_{r\sigma,p_{r}}\}=J_{r\sigma}$;

\item $I'_{r,i}\gamma_{r} = \left\{
\begin{array}{ll}
I_{r\sigma,i} & \text{if } \tau_{\beta}(J'_{r})=\tau_{\alpha}(J_{r\sigma}) \\
I_{r\sigma,p_{r}-i+1} & \text{if } \tau_{\beta}(J'_{r})=\tau_{\alpha}(J_{r\sigma})^R,
\end{array}
\right.$

for $i=1,\ldots,p_{r}$;

\item $J_{r\sigma,i}\delta_{r} = \left\{
\begin{array}{ll}
J'_{r,i} & \text{if } \tau_{\beta}(J'_{r})=\tau_{\alpha}(J_{r\sigma}) \\
J'_{r,p_{r}-i+1} & \text{if } \tau_{\beta}(J'_{r})=\tau_{\alpha}(J_{r\sigma})^R,
\end{array}
\right.$

for $i=1,\ldots,p_{r}$;

\item $\gamma_{r}|_{I'_{r,i}}$ is $\left\{
\begin{array}{ll}
\text{order-preserving} &  \text{if (a) or (b) is satisfied} \\
\text{order-reversing} &  \text{otherwise},
\end{array}%
\right.$\\
where
\begin{description}
  \item (a) $\tau_{\beta}(J'_{r}) = \tau_{\alpha}(J_{r\sigma})$ and $\alpha |_{I_{r\sigma,i}}$ and $\beta |_{I'_{r,i}}$ are both order-preserving or
both order-reversing, and
  \item (b) $\tau_{\beta}(J'_{r}) = \tau_{\alpha}(J_{r\sigma})^{R}$ and $\alpha |_{I_{r\sigma,p_{r}-i+1}}$ is order-preserving and $\beta |_{I'_{r,i}}$  is
order-reversing or vice versa,
\end{description}
for $i=1,\ldots,p_{r}$;

\item $\delta_{r}|_{J_{r\sigma}}$ is $\left\{
\begin{array}{ll}
\text{order-preserving} & \text{if } \tau_{\beta}(J'_{r}) = \tau_{\alpha}(J_{r\sigma}) \\
\text{order-reversing} & \text{if } \tau_{\beta}(J'_{r}) = \tau_{\alpha}(J_{r\sigma})^{R}.
\end{array}
\right.$
\end{itemize}

It is easy to verify that both $\gamma_{r}$ and $\delta_{r}$ are well defined.
Then, we define partial transformations $\gamma$ and $\delta$ as
following:

\begin{itemize}
\item $\dom\gamma = \cup\{I'_{1,1},\ldots, I'_{1,p_{1}},\ldots, I'_{k,1}, \ldots, I'_{k,p_{k}}\}=\dom\beta$;

\item $\dom\delta = \cup\{J_{1\sigma,1},\ldots, J_{1\sigma,p_{1}},\ldots, J_{k\sigma,1},\ldots, J_{k\sigma,p_{k}}\}=\im\alpha$;

\item $\gamma |_{I'_{r,s}} = \gamma_{r} |_{I'_{r,s}}$ for $r=1,\ldots,k$ and $s=1,\ldots,p_{r};$

\item $\delta |_{J_{r\sigma,s}} = \delta_{r} |_{J_{r\sigma,s}}$ for $r=1,\ldots,k$ and $s=1,\ldots,p_{r}$.
\end{itemize}

Clearly, both transformations $\gamma$ and $\delta$ are partial automorphisms. Let $r=1,\ldots,k$ and $s=1,\ldots,p_{r}$.
Then
$$
I'_{r,s}\gamma \alpha \delta =\left\{
\begin{array}{ll}
I_{r\sigma,s}\alpha\delta = J_{r\sigma,s}\delta = J'_{r,s} = I'_{r,s}\beta & \text{if } \tau_{\beta}(J'_{r}) = \tau_{\alpha}(J_{r\sigma}) \\
I_{r\sigma,p_{r}-s+1}\alpha\delta = J_{r\sigma, p_{r}-s+1}\delta = J'_{r,s} = I'_{r,s}\beta & \text{if } \tau_{\beta}(J'_{r}) = \tau_{\alpha}(J_{r\sigma})^{R}.
\end{array}\right.
$$
Taking into account (a) and (b), we can deduce that $\gamma \alpha \delta |_{I'_{r,s}}$ is order-preserving if $\beta |_{I'_{r,s}}$ is
order-preserving and $\gamma \alpha \delta |_{I'_{r,s}}$ is order-reversing if $\beta |_{I'_{r,s}}$ is order-reversing, which allows us to conclude
that $\beta = \gamma \alpha \delta$.
On the other hand,
since $\im\gamma = \dom\alpha$ and $\im\alpha = \dom\delta$, we obtain
$\gamma^{-1}\gamma\alpha\delta\delta^{-1} = \id|_{\dom\alpha}\alpha\id|_{\im\alpha} = \alpha$
and so we also have $\alpha = \gamma^{-1}\beta\delta^{-1}$, as required.
\end{proof}

Now, we can describe the Green's relation $\cal J$ for the monoid $\IEnd(P_n)$.

\begin{proposition}\label{prGR3}
Let $\alpha, \beta \in \IEnd(P_n)$. Then $\alpha {\cal J} \beta$ if and only if $\alpha$ and $\beta$ have similar type.
\end{proposition}
\begin{proof}
Let $\alpha, \beta \in \IEnd(P_n)$ be such that $\alpha {\cal J} \beta$. Then, there exists $\gamma \in \IEnd(P_n)$ such that $\alpha {\cal L} \gamma {\cal R} \beta$ and so, by Propositions \ref{prGR1} and \ref{prGR2}, we have $\dom\gamma=\dom\beta$ and $\im\gamma=\im\alpha$ and
$\alpha\gamma^{-1}, \gamma^{-1}\beta \in \PAut(P_n)$. In addition, $\alpha^{-1}\alpha\gamma^{-1}=\gamma^{-1} = \gamma^{-1}\beta\beta^{-1}$.
Moreover, $\dom(\gamma^{-1}\beta)=\im\alpha$ and $\im(\gamma^{-1}\beta)=\im\beta$. Hence, $\gamma^{-1}\beta\in\PAut(P_n)$
maps each maximal interval $J$ of $\im\alpha$ into a maximal interval $J\gamma^{-1}\beta$ of $\im\beta$, thus defining a bijection $\sigma$ ($J\mapsto J\sigma= J\gamma^{-1}\beta$) from the set of maximal intervals of $\im\alpha$ into the set of maximal intervals of $\im\beta$.
Let $J$ be a maximal interval of $\im\alpha$.
Then $(J\sigma)\beta^{-1}=J\gamma^{-1}\beta\beta^{-1}=J\gamma^{-1}=J\alpha^{-1}\alpha\gamma^{-1}=(J\alpha^{-1})\alpha\gamma^{-1}$.
Since $\alpha\gamma^{-1}\in\PAut(P_n)$, we may deduce that $\tau_{\alpha}(J)\in \{\tau_{\beta}(J\sigma), \tau_{\beta}(J\sigma)^{R}\}$.
Therefore $\alpha$ and $\beta$ have similar type.

Conversely, let $\alpha, \beta \in \IEnd(P_{n})$ be such that $\alpha$ and $\beta$ have
similar type. Then,  by Lemma \ref{leGR1}, we have directly $\alpha \mathcal{J} \beta$, as required.
\end{proof}

We finish this section with the description of the Green's relation $\mathcal{J}$ of $\PAut(P_n)$, which follows immediately from Lemma \ref{leGR1} and
Proposition \ref{prGR3}.

\begin{corollary}
Let $\alpha,\beta \in \PAut(P_n)$. Then $\alpha \mathcal{J} \beta$ if
and only if $\alpha$ and $\beta$ have similar type.
\end{corollary}

Observe that the type of a maximal interval of the image of an element of $\PAut(P_n)$ is always a unitary sequence which we can identify with the size of the interval taken. Therefore, two elements $\alpha$ and $\beta$ of $\PAut(P_n)$  have similar type if and only if
they have maximal intervals of their images with the same size and with the same number of occurrences.

\section{Cardinality}

Let $n\in \mathbb{N}$ and $\overline{n}=\{1,\ldots,n\}$.
We will determine the cardinality of $\PAut(P_n)$ as well as of $\IEnd(P_n)$. For this, we need some technical notations.

Let $A\in \{0,1\}^{n}$ and let $A(p)$ denotes the element on the position $p$ in $A$. Further, let $A(0)=A(n+1)=0$.

Let $R_{A}=\{p\in \overline{n} \mid A(p-1)=0$ and $A(p)=1\}$ and $r_{A}=\left\vert R_{A}\right\vert$.

Let $s_{A}=\underset{p=1}{\overset{n}{\sum}}A(p)$.

Let $z(1)=1$, $z(2)=r_{A}$ and
$$
q_{A,i}=\left\{
\begin{array}{cl}
\left(
\begin{array}{c}
n-s_{A}+z(i) \\
r_{A}%
\end{array}%
\right)  & \text{if  $A(p)\neq 0$, for some $p \in \overline{n}$} \\
1 & \text{otherwise},
\end{array}
\right.
$$
for $i\in \{1,2\}$.

Let $t_{A,i}=(r_{A}!)q_{A,i}$, for $i\in \{1,2\}$ and let $T_{A}=\left\vert \{p\in R_{A} \mid A(p+1)=1\}\right\vert$.

\begin{theorem}\label{prC1}
One has
~$\left\vert \PAut(P_{n})\right\vert ={\!\!\!\!\displaystyle\sum_{A\in \{0,1\}^{n}} 2^{T_{A}}t_{A,1}}$~
and
~$\left\vert \IEnd(P_{n})\right\vert ={\!\!\!\!\displaystyle\sum_{A\in \{0,1\}^{n}} 2^{T_{A}}t_{A,2}}$.
\end{theorem}

\begin{proof}
The domain of an injective endomorphism on $P_{n}$ is a subset of $\overline{n}$. For each $A\in \{0,1\}^{n}$,
let $A^{\ast}$ be the subset of $\overline{n}$  with $x\in A^{\ast}$ if and only if $A(x)=1$. In
particular, by $A\longmapsto A^{\ast}$, a bijection between $\{0,1\}^{n}$
and the powerset of $\overline{n}$, i.e. between $\{0,1\}^{n}$ and the
possible domains of injective endomorphisms on $P_{n}$, is given. Let $A\in \{0,1\}^{n}$.

\smallskip

First, we suppose that $A\neq (0,0,\ldots,0)$. Then $A^{\ast}$ consists of $r_{A}$ maximal intervals $A_{1}<A_{2}<\cdots <A_{r_{A}}$
of $A^{\ast}$. For $i\in \{1,\ldots,r_{A}\}$, let $p_{i}$ be the minimal element in the set $A_{i}$. So, we
have $A(p_{i}-1)=0$ and $A(p_{i})=1$, for $i\in \{1,\ldots,r_{A}\}$.
This provides $R_{A}=\{p_{i} \mid i\in \{1,\ldots,r_{A}\}\}$.
Moreover, we have $s_{A}=\left\vert A^{\ast}\right\vert$.

An injective endomorphism on $P_{n}$ with domain $A^{\ast}$ has the
form
$$
\left(
\begin{array}{cccc}
A_{1} & A_{2} & \cdots  & A_{r_{A}} \\
B_{1} & B_{2} & \cdots  & B_{r_{A}}%
\end{array}%
\right),
$$
where $B_{1}, \ldots, B_{r_{A}}$ are intervals. We observe that for each permutation $\sigma$ on $\{1,\ldots,r_{A}\}$, there is a possible
image sequence $B_{1}, \ldots, B_{r_{A}}$ such that $B_{1\sigma} < B_{2\sigma} < \cdots < B_{r_{A}\sigma}$, i.e. there are $r_{A}!$ possibilities in which the
intervals $B_{1}, \ldots, B_{r_{A}}$ are ordered.
If the image sequence $B_{1}, \ldots, B_{r_{A}}$ is ordered by $B_{1\sigma} < B_{2\sigma} < \cdots < B_{r_{A}\sigma}$, for some permutation $\sigma$ on
$\{1,\ldots,r_{A}\}$, then there are still $n-s_{A}$ elements being not in the image of an injective endomorphism. If we restricted us to partial
automorphisms then there are $b_{1}, \ldots, b_{r_{A}-1} \in \overline{n}$ such that
$B_{1\sigma} < b_{1} < B_{2\sigma} < b_{2} < \cdots < b_{r_{A}-1} < B_{r_{A}\sigma}$ and so there are still $n-s_{A}-r_{A}+1$ elements being not in the image of a partial automorphism.
These remaining elements can be distributed before or after the $B_{i}$'s, i.e. at $r_{A}+1$ places. The number of all these possibilities is
$$\left(
\begin{array}{c}
(r_{A}+1)+(n-s_{A})-1 \\
n-s_{A}%
\end{array}%
\right) =\left(
\begin{array}{c}
r_{A}+n-s_{A} \\
n-s_{A}%
\end{array}%
\right) =\left(
\begin{array}{c}
r_{A}+n-s_{A} \\
r_{A}%
\end{array}%
\right) =q_{A,2}$$
for injective endomorphism and
$$\left(
\begin{array}{c}
(r_{A}+1)+(n-s_{A}-r_{A}+1)-1 \\
n-s_{A}-r_{A}+1%
\end{array}%
\right) =\left(
\begin{array}{c}
n-s_{A}+1 \\
n-s_{A}-r_{A}+1%
\end{array}%
\right) =\left(
\begin{array}{c}
n-s_{A}+1 \\
r_{A}%
\end{array}%
\right) =q_{A,1}$$
if we only consider partial automorphisms.
In other words, there are $q_{A,2}(r_{A}!)=t_{A,2}$ and
$q_{A,1}(r_{A}!)=t_{A,1}$ possibilities for the intervals $B_{1},\ldots,B_{r_{A}}$, whenever $A^{\ast}$ (with the
partition $A_{1}<\cdots <A_{r_{A}})$ is the domain of an injective endomorphism and of a partial automorphism, respectively.
For $i\in \{1,\ldots,r_{A}\}$, if $\left\vert A_{i}\right\vert \geq 2$ then we have to consider two cases, namely
$\left(\begin{array}{c}
A_{i} \\
B_{i}%
\end{array}%
\right)$ is order-preserving or order-reversing. In order to realize it, we consider the cardinality $T_{A}$ of the set
$D_{A}=\{i\in \{1,\ldots,r_{A}\} \mid \left\vert A_{i}\right\vert \geq 2\}$, i.e. $T_{A}=\left\vert D_{A}\right\vert$. So, we have still to consider $2^{T_{A}}$
possibilities, whenever the intervals $B_{1},\ldots,B_{r_{A}}$ are already fixed. Observe that $D_{A}=\{p\in R_{A} \mid A(p+1)=1\}$.

Thus, there are $2^{T_{A}}t_{A,2}$ injective endomorphisms and $2^{T_{A}}t_{A,1}$ partial automorphisms on $P_{n}$ with
domain $A^{\ast}$.

\smallskip

Next, suppose that $A=(0,0,\ldots,0)$. Then, there exists exactly one injective endomorphism on $P_{n}$ with the domain $A^{\ast}=\emptyset$,
namely the empty transformation. In this case, we have $q_{A,1}=q_{A,2}=1$ and $r_{A}=T_{A}=0$. Hence,
$t_{A,1}=t_{A,2}=q_{A,1}(r_{A}!)=1(0!)=1$, $2^{T_{A}}=2^{0}=1$ and $2^{T_{A}}t_{A,1}=2^{T_{A}}t_{A,2}=1$.

\smallskip

Now, we conclude that
$\left\vert \PAut(P_{n})\right\vert=\!\sum_{A\in \{0,1\}^{n}} 2^{T_{A}}t_{A,1}$ and
$\left\vert\IEnd(P_{n})\right\vert=\!\sum_{A\in\{0,1\}^{n}} 2^{T_{A}}t_{A,2}$, as required.
\end{proof}

\section{Generators and Rank}

In this section we present the main results of this paper. We are referring to the calculation of the ranks of $\PAut(P_n)$ and $\IEnd(P_n)$.
In both cases, we proceed by determining a generating set of minimal size.

\smallskip

Clearly, $\PAut(P_{1}) = \{\id,\emptyset\}$ is a generating set of minimal size of $\PAut(P_{1})=\IEnd(P_{1})$, where $\emptyset$ is the empty transformation.
Moreover, it is easy to verify that
$$\mathcal{G} = \left\{\left(
\begin{array}{cc}
1 & 2 \\
2 & 1%
\end{array}%
\right),
\left(
\begin{array}{c}
1 \\
1%
\end{array}%
\right)\right\}$$
is a generating set of minimal size of
$$\PAut(P_{2})=\IEnd(P_{2})=
\left\{ \id,
\left(\begin{array}{cc}
1 & 2 \\
2 & 1%
\end{array}%
\right),
\left(\begin{array}{c}
1 \\
1%
\end{array}%
\right),
\left(\begin{array}{c}
1 \\
2%
\end{array}%
\right),
\left(\begin{array}{c}
2 \\
1%
\end{array}%
\right),
\left(\begin{array}{c}
2 \\
2%
\end{array}%
\right),
\emptyset \right\}.$$
This shows that
$$\rank\PAut(P_{1})=\rank\IEnd(P_{1})=\rank\PAut(P_{2})=\rank\IEnd(P_{2})=2.$$

Next, let $n\geq 3$ and define
$$
\tau = \left(
             \begin{array}{ccccc}
               1 & 2 & \cdots & n-1 & n \\
               n & n-1 & \cdots & 2 & 1 \\
             \end{array}
           \right)
$$
and
$$\alpha_i = \left(
             \begin{array}{ccccccccc}
               1 & 2 & \cdots & i-1 & i+1 & i+2 & \cdots & n-1 & n \\
               1 & 2 & \cdots & i-1 & n & n-1 & \cdots & i+2 & i+1 \\
             \end{array}
           \right),
$$
for $i = 1,2,\ldots,n$.

Let
$$
\mathcal{A} = \left\{
\begin{array}{ll}
\{\tau, \alpha_{1}, \alpha_{2}\} & \text{if } n=3 \\
\{\tau\} \cup \{\alpha_{i} \mid i=1,2,\ldots,n-2\} & \text{if } n \geq 4.
\end{array}%
\right.
$$

\smallskip

First, we will show that $\cal A$ is a generating set of $\PAut(P_{n})$.
To accomplish this aim we start by proving a series of lemmas.

\begin{lemma}\label{leG1}
One has $\{\alpha_i \mid i = n-1,n\} \subseteq \langle \cal A \rangle$.
\end{lemma}
\begin{proof}
The proof follows immediately from the relations $\alpha_i = \tau\alpha^2_{n-i+1}\tau$, for $i = n-1,n$.
\end{proof}

Let
$$\alpha_i^* = \left(
             \begin{array}{ccccccccc}
               1 & 2 & \cdots & i-2 & i-1 & i+1 & \cdots & n-1 & n \\
               i-1 & i-2 & \cdots & 2 & 1 & i+1 & \cdots & n-1 & n \\
             \end{array}
           \right),$$
for $i = 1,2,\ldots,n$.

\begin{lemma}\label{leG2}
One has $\alpha_i^* \in \langle \cal A \rangle$, for $i = 1,2,\ldots,n$.
\end{lemma}
\begin{proof}
We have $\alpha_i^* = \alpha_i\tau\alpha_{n-i+1}\tau\alpha_i$,
whence $\alpha_i^* \in \langle \cal A \rangle$, for $i = 1,2,\ldots,n$.
\end{proof}

Let
$$\varepsilon_{i,j} = \left(
             \begin{array}{cccccccccc}
               1 & 2 & \cdots & i-1 & i+1 & \cdots & j-1 & j+1 & \cdots & n \\
               1 & 2 & \cdots & i-1 & i+1 & \cdots & j-1 & j+1 & \cdots & n \\
             \end{array}
           \right),
$$
for $1 \leq i < i+1 < j \leq n$.

\begin{lemma}\label{leG3}
One has $\varepsilon_{i,j} \in \langle \cal A \rangle$, for $1 \leq i < i+1 < j \leq n$.
\end{lemma}
\begin{proof}
We have $\varepsilon_{i,j} = \alpha_i^2\alpha_j^2$, whence
$\varepsilon_{i,j} \in \langle \cal A \rangle$, for $1 \leq i < i+1 < j \leq n$.
\end{proof}

Let
$$\varepsilon_{i,j}^* = \left(
             \begin{array}{cccccccccccc}
               1 & 2 & \cdots & i-1 & i+1 & i+2 & \cdots & j-2 & j-1 & j+1 & \cdots & n \\
               1 & 2 & \cdots & i-1 & j-1 & j-2 & \cdots & i+2 & i+1 & j+1 & \cdots & n \\
             \end{array}
           \right),$$
for $1 \leq i < i+1 < j \leq n$.

\begin{lemma}\label{leG4}
One has $\varepsilon_{i,j}^* \in \langle \cal A \rangle$, for $1 \leq i < i+1 < j \leq n$.
\end{lemma}
\begin{proof}
We have $\varepsilon_{i,j}^* = \varepsilon_{i,j}\alpha_j^*\alpha_{j-i}^*\alpha_j^*$, whence
$\varepsilon_{i,j}^* \in \langle \cal A \rangle$, for $1 \leq i < i+1 < j \leq n$.
\end{proof}

Define
$\alpha_0 = \tau$, $\alpha_{n+1} = \id$, $\varepsilon_{0,n+1}^* = \tau$,
$\varepsilon_{0,j}^* = \alpha_j^*$, for $j=2,\ldots,n$, and $\varepsilon_{i,n+1}^* = \alpha_i$, for $i=1,\ldots,n-1$.

Let
$$
\rho_{i,j}^+ = \left(
             \begin{array}{cccccccccc}
               1 & 2 & \cdots & i-1 & i+2 & \cdots & j & j+2 & \cdots & n \\
               1 & 2 & \cdots & i-1 & i+1 & \cdots & j-1 & j+2 & \cdots & n \\
             \end{array}
           \right),
           $$
for $0 \leq i < i+2 < j \leq n$.

\begin{lemma}\label{leG5}
One has $\rho_{i,j}^+ \in \langle \cal A \rangle$, for $0 \leq i < i+2 < j \leq n$.
\end{lemma}
\begin{proof}
We have $\rho_{i,j}^+ = \alpha_i^2\alpha_{i+1}^2\alpha_{j+1}^2\varepsilon_{i,j+1}^*\varepsilon_{i,j}^*$, whence
$\rho_{i,j}^+ \in \langle \cal A \rangle$, for $0 \leq i < i+2 < j \leq n$.
\end{proof}

Let
$$\rho_{i,j}^- = \left(
             \begin{array}{cccccccccccc}
               1 & 2 & \cdots & i-2 & i & \cdots & j-2 & j+1 & \cdots & n \\
               1 & 2 & \cdots & i-2 & i+1 & \cdots & j-1 & j+1 & \cdots & n \\
             \end{array}
           \right),$$
for $1 \leq i < i+2 < j \leq n+1$.

\begin{lemma}\label{leG6}
One has $\rho_{i,j}^- \in \langle \cal A \rangle$, for $1 \leq i < i+2 < j \leq n+1$.
\end{lemma}
\begin{proof}
We have $\rho_{i,j}^- = \alpha_{i-1}^2\alpha_{j-1}^2\alpha_{j}^2\varepsilon_{i-1,j}^*\varepsilon_{i,j}^*$, whence
$\rho_{i,j}^- \in \langle \cal A \rangle$, for $1 \leq i < i+2 < j \leq n+1$.
\end{proof}

Now, we are prepared to prove that $\cal A$ is a generating set of the monoid $\PAut(P_n)$.

\begin{proposition}\label{prG1}
One has $\PAut(P_n) = \langle \cal A \rangle$.
\end{proposition}
\begin{proof}
We will perform this proof by using a recurring construction.
First, for arbitrary element $\alpha$ of $\PAut(P_n)$, we set some notations:  denote by  $I_1^\alpha, I_2^\alpha, \ldots, I_k^\alpha$
the maximal intervals of $\dom \alpha$ such that
$$
I_1^\alpha < I_2^\alpha <\cdots < I_k^\alpha.
$$
Let $J_r^\alpha = I_r^\alpha\alpha$, for $r=1,\ldots,k$.
Then
$J_1^\alpha, J_2^\alpha, \ldots, J_k^\alpha$ are the maximal intervals of $\im \alpha$.
Denote by $\sigma_\alpha$ the permutation of $\{1,2,\ldots,k\}$ such that
$$
J_{1\sigma_\alpha}^\alpha < J_{2\sigma_\alpha}^\alpha < \cdots < J_{k\sigma_\alpha}^\alpha.
$$

\smallskip

Now, fix $\alpha \in \PAut(P_n)$.
Let $I = \overline{n}\setminus \dom \alpha$ and define $\beta = \prod_{i\in I}\alpha_i^2$
(observe that $\alpha_i^2$, $i\in\overline{n}$, is an idempotent and idempotents commute). Clearly, $\dom \beta = \dom \alpha$.

Let $s$ be the least number $r \in \{1,\ldots,k\}$ such that $r\sigma_\alpha \neq r\sigma_\beta$.
Let $t$ be the minimal element in the set $J_{s\sigma_\beta}^\beta$ and $q$ be the maximal element of $J_{s\sigma_\alpha\sigma_\beta}^\beta$.
Then, we put
$$
\beta = \beta\varepsilon^*_{t-1,q+1}
$$
(i.e. we define a \textit{new} $\beta$ as being $\beta\varepsilon^*_{t-1,q+1}$; below we will made similar \textit{variables}'s redefinitions).
Then  either $r\sigma_\alpha = r\sigma_\beta$, for all $r \in \{1,\ldots,k\}$,
or the least number $r\in \{1,\ldots,k\}$ such that $r\sigma_\alpha \neq r\sigma_\beta$ is greater than $s$.

We repeat the procedure until $r\sigma_\alpha = r\sigma_\beta$ for all $r \in \{1,\ldots,k\}$.

Further, we put $\gamma = \beta$ and let $u$ be the least number $p \in \{1,\ldots,k\}$ such that
$\gamma|_{I_{p\sigma_\gamma}^\alpha} \neq \alpha|_{I_{p\sigma_\gamma}^\alpha}$.

If $\im \gamma|_{I_{u\sigma_\gamma}^\alpha} = \im \alpha|_{I_{u\sigma_\gamma}^\alpha}$ then we put
$$
\gamma = \gamma\varepsilon_{a,b}^*,
$$
where $a$ and $b$ are the greatest and the least number, respectively, such that $a < J_{u\sigma_\gamma}^\gamma < b$.

If $\im \gamma|_{I_{u\sigma_\gamma}^\alpha} \neq \im \alpha|_{I_{u\sigma_\gamma}^\alpha}$ then there exist $x, y \in \overline{n}$ such that
$\im \alpha|_{I_{u\sigma_\gamma}^\alpha} = \{x,\ldots,z\}$ and either $\im \gamma|_{I_{u\sigma_\gamma}^\alpha} = \{x-y,\ldots,z-y\}$ or
$\im \gamma|_{I_{u\sigma_\gamma}^\alpha} = \{x+y,\ldots,z+y\}$, where $z = x + |I_{u\sigma_\gamma}^\alpha| - 1$.

First, suppose that $\im \gamma|_{I_{u\sigma_\gamma}^\alpha} = \{x-y,\ldots,z-y\}$. Then,
there exists $j\in \overline{n}$ with $j > J_{u\sigma_\gamma}^\gamma$ such that $j,j+1 \notin \im \gamma$. In this case, we put
$$
\gamma = \gamma\rho_{x-y,j+1}^-.
$$

On the other hand, admit that $\im \gamma|_{I_{u\sigma_\gamma}^\alpha} = \{x+y,\ldots,z+y\}$.
Then, there exists $j < J_{u\sigma_\gamma}^\gamma$, with $j > J_{p\sigma_\gamma}^\gamma$, for all $p < u$ such that $j-1,j \notin \im \gamma$.
In this case, we put
$$\gamma = \gamma\rho_{j-1,z+y}^+.$$

After $y$ such steps, we obtain a transformation $\gamma$ such that $\im \gamma|_{I_{u\sigma_\gamma}^\alpha} = \im \alpha|_{I_{u\sigma_\gamma}^\alpha}$.
If $\gamma|_{I_{u\sigma_\gamma}^\alpha} \ne \alpha|_{I_{u\sigma_\gamma}^\alpha}$
then we put
$$
\gamma = \gamma\varepsilon_{a,b}^*,
$$
where $a$ and $b$ are the greatest and the least number, respectively, such that $a < J_{u\sigma_\gamma}^\gamma < b$.

We repeat the procedure until $\gamma = \alpha$. Therefore, by Lemmas \ref{leG1}-\ref{leG6},
we may deduce that $\alpha \in \langle \cal A \rangle$ and so $\cal A$ is a generating set of $\PAut(P_n)$, as required.
\end{proof}

\smallskip

Next, we will show that $\cal A$ is a generating set of $\PAut(P_n)$ of minimal size.

\smallskip

Let $G$ be a generating set of $\PAut(P_n)$.

First, notice that $\dom\tau = \overline{n}$. Moreover, for $\alpha\in \PAut(P_n)$, clearly, we have
$\dom\alpha =\overline{n}$ if and only if $\alpha=\tau$ or $\alpha=\id=\tau^2$.
Thus, it follows immediately that:

\begin{lemma}\label{leR1}
One has $\tau \in G$.
\end{lemma}

Let
$$
A_i = \{\alpha \in \PAut(P_n) \mid \dom \alpha = \overline{n}\setminus\{i\} \mbox{ or } \dom \alpha = \overline{n}\setminus\{n-i+1\}\},
$$
for $i = 1,\ldots,\lceil\frac{n}{2}\rceil$.

\begin{lemma}\label{leR2}
One has $|G \cap A_i| \geq 1$, for all $i \in \{1,\ldots,\lceil\frac{n}{2}\rceil\}$.
\end{lemma}
\begin{proof}
Let $i \in \{1,\ldots,\lceil\frac{n}{2}\rceil\}$ and consider the transformation $\alpha_{i}$ defined previously.
Notice that $\dom \alpha_i = \overline{n}\setminus\{i\}$ and so $\alpha_i\in A_i$.
Let $\beta_1,\ldots,\beta_k\in G\setminus\{\id\}$ be such that
$\alpha_{i}=\beta_1\cdots\beta_k$ and $\{\beta_j,\beta_{j+1}\}\ne\{\tau\}$, for $j=1,\ldots,k-1$.
Since $\dom \alpha_i = \dom (\beta_1\cdots\beta_k) \subseteq \dom \beta_1$, $\rank\alpha_i=n-1$ and $\beta_1\neq\id$,
we have $\dom\beta_1=\dom\alpha_i$ or $\beta_1=\tau$.

If $\dom\beta_1=\dom\alpha_i$ then $\beta_1\in A_i$ and so $\beta_1\in G\cap A_i$.

On the other hand, suppose that $\beta_1=\tau$. In that case,
since $\dom \alpha_i = \dom(\beta_1\cdots\beta_k) \subseteq \dom(\tau\beta_2)$,
$\rank\alpha_i=n-1$ and $\beta_{2}\in G\setminus \{\id,\tau\}$, we have
$\dom\alpha_i=\dom(\tau\beta_2)$, whence $\dom \beta_2 = \overline{n}\setminus\{n-i+1\}$ and so  $\beta_2\in G\cap A_i$.

Thus, in both cases, we have shown that $|G \cap A_i| \neq\emptyset$, as required.
\end{proof}

\begin{lemma}\label{leR3}
Let $n \geq 6$. Then $|G \cap A_i| \geq 2$, for all $i \in \{3,\ldots,\lfloor\frac{n}{2}\rfloor\}$.
\end{lemma}
\begin{proof}
First, observe that it is a routine matter to check that $|A_i| = 16$, for all $i \in \{3,\ldots,\lfloor\frac{n}{2}\rfloor\}$.
Recall also that $\tau \in G$, by Lemma \ref{leR1}.

Now, assume by contradiction that $|G \cap A_i| < 2$, for some $i \in \{3,\ldots,\lfloor\frac{n}{2}\rfloor\}$.
Then, by Lemma \ref{leR2}, we have $G \cap A_i = \{\alpha\}$, for some
$\alpha\in\PAut(P_n)$. Without loss of generality, we may suppose that $\dom \alpha = \overline{n} \setminus \{i\}$.
Hence, we have two cases:

Case 1. $\im \alpha = \overline{n} \setminus \{i\}$. Then, as $\alpha^3 = \alpha$ and $\rank \alpha\tau\alpha = n-2$, we have
$$
\langle G \rangle \cap A_i = \{\alpha, \alpha^2, \alpha\tau, \tau\alpha, \alpha^2\tau, \tau\alpha^2, \tau\alpha\tau, \tau\alpha^2\tau\}\neq A_i,
$$
which is a contradiction (since $G$ is a generating set of $\PAut(P_n)$).

Case 2. $\im \alpha = \overline{n} \setminus \{n-i+1\}$. In this case,
as $(\alpha\tau)^2 = \id|_{\dom\alpha}$, $(\tau\alpha)^2 = \id|_{\im\alpha}$ and $\rank \alpha^2 = n-2$, we obtain
$$
\langle G \rangle \cap A_i = \{\alpha, \alpha\tau, \tau\alpha, \alpha\tau\alpha, \tau\alpha\tau, (\alpha\tau)^2, (\tau\alpha)^2, \tau(\alpha\tau)^2\} \neq A_i,
$$
which again is a contradiction, as required.
\end{proof}

Now, as a consequence of Proposition \ref{prG1} and Lemmas \ref{leR1}\---\ref{leR3},
we may prove the first of our main results:

\begin{theorem}\label{thR1}
The rank of $\PAut(P_3)$ is equal to $3$ and, for $n\geq 4$, the rank of $\PAut(P_n)$ is equal to $n-1$.
\end{theorem}

\begin{proof}
By Proposition \ref{prG1}, the set $\cal A$ generates $\PAut(P_{n})$. Thus,
$$\rank\PAut(P_{n})\leq \left\vert \mathcal{A} \right\vert = \left\{
\begin{array}{ll}
3 & \text{if } n=3 \\
n-1 & \text{if } n\geq 4.
\end{array}%
\right.$$

Let $G$ be any generating set of $\PAut(P_{n})$. By Lemmas \ref{leR1} and \ref{leR2}, the transformation $\tau$ and $\left\lceil \frac{n}{2}\right\rceil$ pairwise
different transformations of rank $n-1$ are in $G$. Thus, $\left\vert G\right\vert \geq 1+\left\lceil \frac{n}{2}\right\rceil$. In particular, we
have $\left\vert G\right\vert \geq 3$, if $n=3,4$, and $\left\vert G\right\vert \geq 4$, if $n=5$. If $n\geq 6$ then, by Lemma \ref{leR3}, there exist
$\left\lfloor \frac{n}{2}\right\rfloor -2$ additional pairwise different transformations with rank $n-1$ in $G$.
This shows that
$$
\rank\PAut(P_{n}) \geq \left\{
\begin{array}{ll}
3 & \text{if } n=3 \\
n-1 & \text{if } n\geq 4,
\end{array}%
\right.
$$
as required.
\end{proof}

\smallskip

Next, we calculate the rank of the monoid $\IEnd(P_{n})$.

\smallskip

Define
$$
\beta_i = \left(
             \begin{array}{ccccccccc}
               1 & 2 & \cdots & i-1 & i+1 & i+2 & \cdots & n-1 & n \\
               1 & 2 & \cdots & i-1 & i & i+1 & \cdots & n-2 & n-1 \\
             \end{array}
           \right),
           $$
for  $i = 2,\ldots,n-1$, and let $\mathcal{B} = \mathcal{A} \cup \{\beta_i \mid i = 2,\ldots,\lceil\frac{n}{2}\rceil\}$.

\begin{lemma}\label{leG7}
One has $\{\beta_i \mid i = 2,\ldots,n-1\} \subseteq \langle \cal B \rangle$.
\end{lemma}
\begin{proof}
For $i = 2,\ldots,\lceil\frac{n}{2}\rceil$, we have $\beta_i \in \cal B$.
Let $i = \lceil\frac{n}{2}\rceil+1, \ldots, n-1$ then $\beta_i = \tau\beta_{n-i+1}\alpha_n^*$.
\end{proof}

\begin{proposition}\label{prG2}
Let $\beta \in \IEnd(P_n)\setminus \PAut(P_n)$. Then $\beta \in \langle \PAut(P_n) \cup \{\beta_i \mid i = 2,\ldots,n-1\} \rangle$.
\end{proposition}
\begin{proof}
Let $\beta \in \IEnd(P_n)\setminus \PAut(P_n)$.
Then, it is easy to show that there exists a transformation $\delta \in \PAut(P_n)$ such that $\dom \delta = \im\beta$ and
$\im \delta \subseteq \{1,2,\ldots,|\im\beta|+\text{m}_\beta-1\}$, where $\text{m}_\beta$ is the number of the maximal intervals of $\im\beta$.

Define $\bar{\beta} = \beta\delta$. It is clear that $\dom\bar{\beta} = \dom\beta$.

Further, let $I$ be the set of all $x \in \dom \bar{\beta}$ such that
$x\bar{\beta}+1 \in \im\bar{\beta}$ and $x\bar{\beta}+1 \notin \{(x-1)\bar{\beta}, (x+1)\bar{\beta}\}$.
Clearly, $I \neq \emptyset$ since $\beta \notin \PAut(P_{n})$.
We let $I= \{i_1, \ldots, i_k\}$ be such that $i_1\bar{\beta} < i_2\bar{\beta} < \cdots < i_k\bar{\beta}$.

Let $X_1, X_2, \ldots, X_{k+1}$ be the partition of $\dom\bar{\beta}$ such that
$$
X_r = \{x \in \dom \bar{\beta} \mid i_{r-1}\bar{\beta} < x\bar{\beta} \leq i_r\bar{\beta}\},
$$
for $r \in \{1,2,\ldots,k+1\}$, where $i_0\bar{\beta} =0$ and $i_{k+1}\bar{\beta} = n$.

Let $\beta^*$ be the transformation defined by
$x\beta^* = x\bar{\beta} + r - 1$, for all $x \in X_r$ and $r = 1,2,\ldots,k+1$.
It is clear that $\dom\beta^* = \dom\bar{\beta} = \dom\beta$ and $\beta^* \in \IEnd(P_n)$.

First, we show $\beta^* \in \PAut(P_n)$.
Let $u\in\overline{n}$ be such that $u, u+1 \in \im \beta^*$.
Then there exist $a \in X_{r_1}$ and $b \in X_{r_2}$ such that $a\bar{\beta} + r_1 - 1 = u$ and $b\bar{\beta} + r_2 - 1 = u+1$,
for some $r_1, r_2 \in \{1,\ldots,k+1\}$. In order to obtain a contradiction, assume that $r_1 \neq r_2$.

Suppose that $r_1 < r_2$. Then $X_{r_1}\bar{\beta} < X_{r_2}\bar{\beta}$ and so $a\bar{\beta} < b\bar{\beta}$.
This implies $r_1 +1 \leq r_2$ and $a\bar{\beta} + 1 \leq b\bar{\beta}$, whence
$b\bar{\beta} + r_2 \geq a\bar{\beta} + 1 + r_1 + 1 = a\bar{\beta} + r_1 - 1 + 3 = u + 3 = b\bar{\beta} + r_2 - 1 + 2 = b\bar{\beta} + r_2 + 1$.
Thus $b\bar{\beta} \geq b\bar{\beta} + 1$, which is a contradiction.

On the other hand, suppose that $r_1 > r_2$. Then $X_{r_1}\bar{\beta} > X_{r_2}\bar{\beta}$ and so $a\bar{\beta} > b\bar{\beta}$.
This implies $r_1 \geq r_2+1$ and $a\bar{\beta} \geq b\bar{\beta} + 1$. Thus, we have
$a\bar{\beta} + r_1 \geq b\bar{\beta} + 1 + r_2 + 1 = b\bar{\beta} + r_2 - 1 + 3 = u + 1 + 3 = a\bar{\beta} + r_1 + 3$, whence
$a\bar{\beta} \geq a\bar{\beta} + 3$, which is a contradiction.

Therefore, we have $r_1 = r_2$. Then $u = a\bar{\beta} + r_1 - 1$, $u+1 = b\bar{\beta} + r_1 - 1$ and so $a, b \in X_{r_1}$.
This implies $b\bar{\beta} + r_1 = b\bar{\beta} + r_1 - 1 + 1 = u+1+1 = a\bar{\beta} + r_1 - 1 + 2 = a\bar{\beta} + r_1 + 1$, i.e.
$b\bar{\beta} = a\bar{\beta} + 1$. Thus $a\bar{\beta} + 1 \in \im \bar{\beta}$, since $b \in \dom\bar{\beta}$.
Assume $a\bar{\beta} + 1 \notin \{(a-1)\bar{\beta}, (a+1)\bar{\beta}\}$.
Then $a \in I$ and so $a=i_{r_1}$ and $b\bar{\beta} \leq a\bar{\beta}$, since $a, b \in X_{r_1}$.
Hence, $a\bar{\beta} + 1 = b\bar{\beta} \leq a\bar{\beta}$, which is a contradiction.
Thus, $b\bar{\beta} = a\bar{\beta} + 1 \in \{(a-1)\bar{\beta}, (a+1)\bar{\beta}\}$ and so we obtain $b \in \{a-1,a+1\}$.

This shows that $\beta^* \in \PAut(P_n)$.

Finally, we show that
$
\beta = \beta^*\beta_{i_1\bar{\beta}+1}\beta_{i_2\bar{\beta}+1}\cdots\beta_{i_k\bar{\beta}+1}\delta^{-1},
$
from which follows that
$$
\beta \in \langle \PAut(P_n) \cup \{\beta_i \mid i = 2,\ldots,n-1\}  \rangle.
$$
Since $\bar{\beta}\delta^{-1} = \beta\delta\delta^{-1} = \beta\id|_{\dom \delta} = \beta\id|_{\im \beta} = \beta$, it suffices to show that
$\bar{\beta} = \beta^*\beta_{i_1\bar{\beta}+1}\beta_{i_2\bar{\beta}+1}\cdots\beta_{i_k\bar{\beta}+1}$.

We proceed by showing that
$$
x\beta^*\beta_{i_1\bar{\beta}+1}\beta_{i_2\bar{\beta}+1}\cdots\beta_{i_s\bar{\beta}+1} =
\left\{\begin{array}{ll}
  x\bar{\beta} & \mbox{if } x \in X_1 \cup \cdots \cup X_s \\
  x\bar{\beta} + r - 1 - s & \mbox{if } x \in X_{r}, \mbox{ for some } r \in \{s+1, \ldots, k+1\},
\end{array}
\right.
$$
by induction on $1 \leq s \leq k$.

Let $s=1$. Then
$$
\begin{array}{rcl}
x\beta^*\beta_{i_1\bar{\beta}+1} &= &(x\bar{\beta} + r - 1)\beta_{i_1\bar{\beta}+1} \\\\
 & =  &
\left\{\begin{array}{ll}
         (x\bar{\beta})\beta_{i_1\bar{\beta}+1} = x\bar{\beta} & \mbox{if } x \in X_1, \mbox{ since } x\bar{\beta} < i_1\bar{\beta}+1 \\
         x\bar{\beta} + r - 1 - 1 = x\bar{\beta} + r - 1 - s  & \mbox{if } x \in X_r  \mbox{ for some } r > 1, \mbox{ since } x\bar{\beta} \geq i_1\bar{\beta}+1.
       \end{array}\right.
\end{array}
$$

Assume that the above expression is true for some $s < k$. We will prove it for $s+1$.

Let $x \in X_1 \cup \cdots \cup X_{s+1}$.

If $x \notin X_{s+1}$ then $x\beta^*\beta_{i_1\bar{\beta}+1}\cdots\beta_{i_s\bar{\beta}+1} = x\bar{\beta}$, by the induction hypothesis and $(x\bar{\beta})\beta_{i_{s+1}\bar{\beta}+1} = x\bar{\beta}$, since $x\bar{\beta} < i_{s+1}\bar{\beta}+1$.

If $x \in X_{s+1}$ then $x\beta^*\beta_{i_1\bar{\beta}+1}\cdots\beta_{i_s\bar{\beta}+1} = x\bar{\beta}+s+1-1-s = x\bar{\beta}$, by the induction hypothesis and $(x\bar{\beta})\beta_{i_{s+1}\bar{\beta}+1} = x\bar{\beta}$, since $x\bar{\beta} < i_{s+1}\bar{\beta}+1$.

Now, let  $x \in X_r$ for some $r \in \{s+2, \ldots, k+1\}$. Then
$$
(x\beta^*\beta_{i_1\bar{\beta}+1}\cdots\beta_{i_s\bar{\beta}+1})\beta_{i_{s+1}\bar{\beta}+1} = (x\bar{\beta}+r-1-s)\beta_{i_{s+1}\bar{\beta}+1} = (x\bar{\beta}+r-1-s)-1 = x\bar{\beta}+r-1-(s+1),
$$
since $x\bar{\beta} > i_{s+1}\bar{\beta}+1$, which completes the proof.
\end{proof}

From Proposition \ref{prG1}, Lemma \ref{leG7} and Proposition \ref{prG2}, we obtain immediately:

\begin{corollary}\label{coG1}
One has $\IEnd(P_n) = \langle \cal B \rangle$.
\end{corollary}

Now, we will prove that $\cal B$ is a generating set of $\IEnd(P_n)$ of minimal size. We start by presenting a series of five lemmas.

Let $G'$ be a generating set of $\IEnd(P_n)$.

\begin{lemma}\label{leR6}
One has $|G' \cap(\IEnd(P_n)\setminus \PAut(P_n))| \geq \lceil\frac{n}{2}\rceil -1$.
\end{lemma}
\begin{proof}
Let $2\leq j \leq n-1$.
Then $\beta_j = \gamma_1\cdots\gamma_k$, for some $k\geq 1$ and $\gamma_1, \ldots, \gamma_k \in G'$.
As $\rank\beta_j = n-1$ then $\rank\gamma_i \geq n-1$, for all $i=1,\ldots,k$, and there exists $i \in \{1,\ldots,k\}$ such that $\gamma_i \notin \PAut(P_n)$. Let $i$ be the least $r \in \{1, \ldots, k\}$ such as $\gamma_r \notin \PAut(P_n)$. Let $\gamma = \gamma_1\cdots\gamma_{i-1} \in \PAut(P_n)$ (with $\gamma = \id$ if $i=1$). Thus, $\beta_j = \gamma\gamma_i\cdots\gamma_k$ implies $\gamma^{-1}\beta_j = \gamma^{-1}\gamma\gamma_i\cdots\gamma_k = \gamma_i\cdots\gamma_k$ (since $\gamma^{-1}\gamma = \id|_{\dom\gamma_i}$).

We have $\rank\gamma = n-1$ or $\rank\gamma = n$. If $\rank\gamma = n-1$ then $\im\gamma^{-1} = \dom\beta_j = \{1,\ldots,n\}\setminus\{j\}$. Thus $\dom\gamma = \{1,\ldots,n\}\setminus\{j\}$, whence $\im\gamma = \dom\beta_j$ or $\im\gamma = \dom\beta_{n-j+1}$, and so $\dom\gamma_i = \im\gamma = \dom\beta_j$ or $\dom\gamma_i = \im\gamma = \dom\beta_{n-j+1}$ (since $\rank\gamma_i = n-1$). If $\rank\gamma = n$ then $\gamma = \id$ or $\gamma = \tau$. If $\gamma = \id$ then $\dom\gamma_i = \dom\beta_j$. If $\gamma = \tau$ then $\dom\gamma_i = \dom\beta_{n-j+1}$. Note that $n-j+1 \geq \lceil\frac{n}{2}\rceil$.

Therefore, we must have in $G'$ at least $\lceil\frac{n-2}{2}\rceil = \lceil\frac{n}{2}\rceil -1$ distinct elements of $\IEnd(P_n)\setminus \PAut(P_n)$.
\end{proof}

\begin{lemma}\label{leR4}
Let $\alpha \in \IEnd(P_n)$ be such that $\dom\alpha \in \{\{1,\ldots,n-1\}, \{2,\ldots,n\}\}$.
Then $\alpha \in \PAut(P_n)$.
\end{lemma}

\begin{proof}
It is a routine matter to verify that
$\alpha \in \{\alpha_{1}, \tau\alpha_{1}, \alpha_{1}\tau, \tau\alpha_{1}\tau, \alpha_{n}, \tau\alpha_{n}, \alpha_{n}\tau, \tau\alpha_{n}\tau\}
\subseteq \PAut(P_{n})$.
\end{proof}

\begin{lemma}\label{leR5}
If $\alpha \in \IEnd(P_n)\setminus \PAut(P_n)$ has rank equal $n-1$ then $\im \alpha \in \{\{1,\ldots,n-1\}, \{2,\ldots,n\}\}$.
\end{lemma}

\begin{proof}
Since $\rank\alpha = n-1$, we can conclude that $\alpha \in \{\beta_{i}, \tau\beta_{i}, \beta_{i}\tau, \tau\beta_{i}\tau \mid i=2,\ldots,n-1\}$.
Let $i=2,\ldots,n-1$.
Since $\im\beta_{i} = \{1,\ldots,n-1\}$ and $\dom\tau = \im\tau = \overline{n}$, we obtain
$$
\im\beta_{i}, \im(\tau\beta_{i}), \im(\beta_{i}\tau), \im(\tau\beta_{i}\tau) \in \{\{1,\ldots,n-1\},\{2,\ldots,n\}\},
$$
whence $\im\alpha \in \{\{1,\ldots,n-1\},\{2,\ldots,n\}\}$, as required.
\end{proof}

\begin{lemma}\label{leR7}
One has $\langle G' \cap \PAut(P_n)\rangle = \PAut(P_n)$.
\end{lemma}
\begin{proof}
First, notice that it is clear that $\tau\in G'$.

On the other hand, let $\alpha$ be any transformation of $\PAut(P_n)$ with $\rank\alpha = n-1$.

Then, there exist $\gamma_1,\ldots,\gamma_k \in G'$ such that
$\alpha=\gamma_1\cdots\gamma_k$ ($k\geq1$). Assume that there exists $i \in \{1,\ldots,k\}$ such that $\gamma_i \notin \PAut(P_n)$. Let $i$ be the least index $r \in \{1,\ldots,k\}$ such that $\gamma_r \notin \PAut(P_n)$ and let $\gamma = \gamma_1\cdots\gamma_{i-1} \in \PAut(P_n)$ (with $\gamma = \id$ if $i=1$). Then $\alpha = \gamma\gamma_i\cdots\gamma_k$ implies $\gamma_i\cdots\gamma_k = \gamma^{-1}\gamma\gamma_i\cdots\gamma_k = \gamma^{-1}\alpha \in \PAut(P_n)$ (since $\gamma^{-1}\gamma = \id|_{\dom\gamma_i}$). Hence, we have $i<k$.
We have $\gamma_{i+1}\cdots\gamma_k \notin \{\id,\tau\}$
(otherwise $\gamma_i = \gamma^{-1}\alpha(\gamma_{i+1}\cdots\gamma_k)^{-1} \in \PAut(P_n)$, which is a contradiction). Hence,
$\rank(\gamma_{i+1}\cdots\gamma_k) = n-1$. Let $\lambda = \gamma_{i+1}\cdots\gamma_k$. Then $\dom \lambda = \im \gamma_i \in \{\{1,\ldots,n-1\},\{2,\ldots,n\}\}$, by Lemma \ref{leR5}. Therefore, we obtain $\lambda \in \PAut(P_n)$, by Lemma \ref{leR4}. Thus,
$\gamma^{-1}\alpha = \gamma_i\lambda$ implies that $\gamma_i = \gamma^{-1}\alpha\lambda^{-1} \in \PAut(P_n)$, which is a contradiction.
Thus, $\gamma_{1},\ldots,\gamma_k \in \PAut(P_n)$.

Therefore, in particular, we showed that $\mathcal{A}\subseteq \langle G' \cap \PAut(P_n)\rangle$ and so $\langle G' \cap \PAut(P_n)\rangle = \PAut(P_n)$,
by Proposition \ref{prG1}.
\end{proof}

Now, as an immediate consequence of Lemma \ref{leR7} and Theorem \ref{thR1}, we have:

\begin{lemma}\label{leR8}
One has $|G' \cap \PAut(P_n)|\geq \left\{
\begin{array}{ll}
3 & \text{if } n=3 \\
n-1 & \text{if } n\geq 4.
\end{array}%
\right.
$
\end{lemma}

Finally, we conclude with the presentation of our second main result.

\begin{theorem}\label{thR2}
The rank of $\IEnd(P_3)$ is equal to $4$ and, for $n \geq 4$, the rank of $\IEnd(P_n)$ is equal to $n + \lceil\frac{n}{2}\rceil-2$.
\end{theorem}

\begin{proof}
By Corollary \ref{coG1}, we have
$$
\rank \IEnd(P_{n}) \leq \left\vert \mathcal{B}\right\vert = \left\{
\begin{array}{ll}
3+1=4 & \text{if } n=3 \\
n-1+\lceil \frac{n}{2} \rceil - 1 = n+\lceil \frac{n}{2} \rceil - 2 & \text{if } n \geq 4.
\end{array}%
\right.
$$
On the other hand, by Lemmas \ref{leR6} and \ref{leR8}, we have
$$
\rank \IEnd(P_{n}) \geq \left\{
\begin{array}{ll}
3+1=4 & \text{if } n=3 \\
n-1+\lceil \frac{n}{2} \rceil - 1 = n+\lceil \frac{n}{2} \rceil - 2 & \text{if } n \geq 4,
\end{array}%
\right.
$$
as required.
\end{proof}

\section*{Acknowledgement}


This work was produced, in part, during the visit of the first and third authors to CMA, FCT NOVA, Lisbon, in July 2019.
The first author was supported by CMA through a visiting researcher fellowship.



{\small \sf
\noindent{\sc Ilinka Dimitrova},
Department of Mathematics,
Faculty of Mathematics and Natural Science,
South-West University "Neofit Rilski",
2700 Blagoevgrad,
Bulgaria;
e-mail: ilinka\_dimitrova@swu.bg.

\medskip

\noindent{\sc V\'\i tor H. Fernandes},
CMA, Departamento de Matem\'atica,
Faculdade de Ci\^encias e Tecnologia,
Universidade NOVA de Lisboa,
Monte da Caparica,
2829-516 Caparica,
Portugal;
e-mail: vhf@fct.unl.pt.

\medskip

\noindent{\sc J\"{o}rg Koppitz},
Institute of Mathematics and Informatics,
Bulgarian Academy of Sciences,
1113 Sofia,
Bulgaria;
e-mail: koppitz@math.bas.bg.

\medskip

\noindent{\sc Teresa M. Quinteiro},
Instituto Superior de Engenharia de Lisboa,
1950-062 Lisboa,
Portugal.
Also:
Centro de Matem\'atica e Aplica\c{c}\~oes,
Faculdade de Ci\^encias e Tecnologia,
Universidade Nova de Lisboa,
Monte da Caparica,
2829-516 Caparica,
Portugal;
e-mail: tmelo@adm.isel.pt.
}

\end{document}